\documentclass[12pt,reqno]{amsart}
\usepackage{graphicx}
\usepackage{amssymb,amsmath}
\usepackage{amsthm}
\usepackage{color}
\usepackage[pdf]{pstricks}
\usepackage{hyperref}
\usepackage{empheq}
\usepackage{pgfplots}

\usepackage{amssymb}
\usepackage{epsfig}
\usepackage{extarrows}
\usepackage{amsfonts}
\usepackage{graphicx}
\usepackage{caption}
\usepackage{subcaption}
\usepackage{tikz}
\usepackage{color}

\newcommand{\R}{\mathbb{R}}

\newtheorem{remark}{Remark}

\newtheorem{assumption}{Assumption}
\newtheorem{definition}{Definition}

\newtheorem{theorem}{Theorem}

\begin{document}
	
\title[Bifurcations of unstable eigenvalues for Stokes waves]{\bf Bifurcations of unstable eigenvalues for Stokes waves derived from conserved energy}

\author{Sergey Dyachenko}
\address[S. Dyachenko]{Department of Mathematics, State University of New York at Buffalo, Buffalo, New York, USA, 14260}
\email{sergeydy@buffalo.edu}
	
\author{Dmitry E. Pelinovsky}
\address[D. E. Pelinovsky]{Department of Mathematics and Statistics, McMaster University, Hamilton, Ontario, Canada, L8S 4K1}
\email{pelinod@mcmaster.ca}
	
\dedicatory{\textit{In memory of Vladimir Zakharov, our friend and mentor, whose vision will continue to inspire us.}}
	
\begin{abstract}
We address Euler's equations for irrotational gravity waves in an infinitely deep fluid rewritten in conformal variables. Stokes waves are traveling waves with the smooth periodic profile. In agreement with the previous numerical results, we give a rigorous proof that the zero eigenvalue bifurcation in the linearized equations of motion for co-periodic perturbations occurs at each extremal point of the energy function versus the steepness parameter, provided that the wave speed is not extremal at the same steepness. We derive the normal form for the unstable eigenvalues and, assisted with numerical approximation of its coefficients, we show that the new unstable eigenvalues emerge only in the direction of increasing steepness.
\end{abstract}

\maketitle

\section{Introduction}

Ocean swell can be viewed in many cases as a train of almost periodic traveling waves propagating along a fixed direction. Understanding stability properties of periodic wave trains are central to wave forecasting. Such periodic traveling waves were originally found by Stokes~\cite{stokes1847theory,stokes1880theory}, and hence they are often referenced as the Stokes waves. Stokes waves are efficiently approximated in the limit of small amplitude~\cite{levi1925determination}. The existence of Stokes waves including the limiting wave with the peaked profile was proven in~\cite{amick1982stokes,plotnikov2002proof,toland1978existence}. 

The stability of Stokes waves is studied either with respect to perturbations co-periodic with the underlying wave (superharmonic), or in a wider space of perturbations periodic with longer periods (subharmonic). In the latter case, the modulational instability, also known as the Benjamin-Feir instability~\cite{benjamin1967disintegration,zakharov1968stability}, is recovered. The modulational instability of Stokes waves was studied rigorously and proven in~\cite{nguyen2020proof,berti2022full,berti2023,hur2023}.
Moreover, the high-frequency instabilities discovered in~\cite{deconinck2011instability} and theoretically studied in~\cite{creedon2021high1,creedon2021high2} have the same modulational nature.

Whether subharmonic or superharmonic, the stability properties of traveling waves are studied in the limit of small amplitude~\cite{creedon2022ahigh} where the small-amplitude expansions offer accurate approximations. 
For the Stokes waves of high steepness and, the limiting Stokes wave with a $2\pi/3$ crest angle, the series expansion diverges and numerical methods are used instead. Numerical solution of the eigenvalue problem for stability of Stokes waves on a surface of an infinitely deep fluid goes back to ~\cite{tanaka1983stability,tanaka1985,longuet1997crest}. In~\cite{murashige2020stability,korotkevich2022superharmonic} the stability problem is treated as an eigenvalue problem for a large matrix in Fourier basis and is restricted to superharmonic perturbations. Recently, it was realized that the stability spectrum can be determined more efficiently via matrix-free methods~\cite{ dyachenko_semenova2022,dyachenko2023quasiperiodic} allowing to extend the stability analysis to nearly limiting Stokes waves~\cite{dyachenko2022almost}, and include the Bloch-Floquet theory to cover subharmonic  perturbations~\cite{deconinck2022instability,DDS2024}.

Figure \ref{fig:hamiltonian} presents a schematic of the dependence of Hamiltonian $H$ (green) and the speed $c$ (red) of the traveling periodic wave  continued with respect to the steepness 
parameter $s$  \cite{korotkevich2022superharmonic,dyachenko_semenova2022,dyachenko2023quasiperiodic,dyachenko2022almost,deconinck2022instability,DDS2024}, see also ~\cite{tanaka1983stability,tanaka1985,longuet1997crest} for the early numerical results suggesting the same behavior of the Hamiltonian and speed versus the steepness. The family of traveling periodic waves bifurcates from the small-amplitude limit $(H_0,c_0)$ at $s = 0$ and oscillates toward $(H_{lim},c_{lim})$ for the limiting Stokes wave with the peaked profile 
and the limiting steepness $s_{lim}$ \cite{amick1982stokes,plotnikov2002proof,toland1978existence}. The striking point of this figure is 
that every extremal point of the Hamiltonian corresponds to the instability bifurcation in the co-periodic stability problem in the sense of the existence of the zero eigenvalue with a higher algebraic multiplicity than the one prescribed by the symmetries of the water wave equations. When the steepness of the periodic wave is increased past the extremal point of the Hamiltonian, a new pair of real eigenvalues bifurcates in the spectrum of the co-periodic stability problem. It is conjectured \cite{DDS2024} that the family of traveling periodic waves displays infinitely many oscillations with infinitely many 
instability bifurcations before reaching the limiting peaked wave. 

\begin{figure}[htb!]
    \centering
    \includegraphics[width=0.95\linewidth]{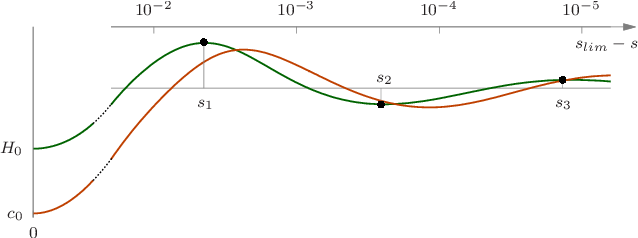}
    \caption{Schematic of oscillations of the Hamiltonian (green) and the speed (red) as the limiting Stokes wave is approached. The figure illustrates $H_{lim}-H$ and $c_{lim}-c$ as a function of $s_{lim}-s$, where $s_{lim}$, $H_{lim}$, and $c_{lim}$ represent the steepness, the Hamiltonian, and the speed of the limiting Stokes wave. The black circles mark the extreme points of the Hamiltonian, where the instability bifurcation occurs.}
    \label{fig:hamiltonian}
\end{figure}

The main purpose of this paper is to give a rigorous proof that the instability bifurcation occurs exactly at each extremal point of the Hamiltonian or, equivalently, the horizontal momentum. If the zero eigenvalue has generally geometric multiplicity two and algebraic multiplicity four due to symmetries of the equations of motion, we show that the zero eigenvalue has geometric multiplicity two and algebraic multiplicity of at least six at the instability bifurcation point. This result is given by Theorem \ref{theorem-crit}. In addition, we compute the normal form for the unstable eigenvalues in the co-periodic stability problem and, assisted with numerical approximation of its coefficients, we show that the new unstable eigenvalues emerge only in the direction of increasing steepness. This result is given by Theorem \ref{theorem-bif}.

For the technical parts of the proofs, we adopt conformal variables for the two-dimensional fluid dynamics developed in~\cite{ovsyannikov1973dynamika,tanveer1991singularities,tanveer1993singularities,babenko1987some,DyachenkoEtAl1996} and used in~\cite{murashige2020stability,korotkevich2022superharmonic,dyachenko_semenova2022,dyachenko2023quasiperiodic} for spectrally accurate numerical approach to the stability problem. The conformal variables allow us to write the problem of finding Stokes wave as a pseudo-differential nonlinear equation~\cite{babenko1987some,locke2024peaked,locke2024smooth}, and formulate the stability problem as a matrix-free pseudo-differential eigenvalue problem with periodic coefficients~\cite{dyachenko_semenova2022,dyachenko2023quasiperiodic}. The conserved quantities of the water wave equations \cite{benjamin1982hamiltonian} are rewritten in conformal variables and impose constraints 
on solutions of the co-periodic stability problem. Computations of the Jordan blocks and Puiseux expansions for multiple eigenvalues are performed in compliance with the constraints, which 
act as the Fredholm solvability conditions for solutions at each order of the perturbation theory. Since 
justification of the Puiseux expansions is fairly known for linear eigenvalue problems \cite{welters2011}, we will focus on actual computations rather than on the justification analysis.

The main result of Theorem \ref{theorem-crit} has been well understood in the dynamics of fluids, based on the numerical results \cite{tanaka1983stability,tanaka1985} and the formal analytical computations \cite{saffman1985,mackay1986stability}. Compared to these earlier works which were based on Zakharov's equations of motion \cite{zakharov1968stability}, we develop the analysis of equations of motion in conformal variables by exploring Babenko's pseudo-differential equation \cite{babenko1987some} and its linearization. We also go beyond the criterion for the instability bifurcation and compute its normal form for the unstable eigenvalues. With the use of much more elaborated numerical computations, we confirm the main pattern that the new unstable eigenvalues arise in the direction of increasing steepness along the family of Stokes waves.

The paper is organized as follows. Equations of motion in physical and conformal variables are written in Section \ref{sec-2}, where we also give the conserved quantities and describe the existence and stability problems for Stokes waves. Section \ref{sec-3} presents the main result on the co-periodic instability bifurcation (Theorem \ref{theorem-crit}). Section \ref{sec-4} presents the normal form for the unstable eigenvalues (Theorem \ref{theorem-bif}). Section \ref{sec-5} contains numerical approximations of eigenfunctions at the instability bifurcation and coefficients of the normal form to confirm the main prediction that every instability bifurcation generates a new unstable eigenvalue in the direction of the increasing steepness. The paper is completed with Section \ref{sec-6} where further questions are discussed.

\section{Equations of motion in conformal variables}
\label{sec-2}

Let $y = \eta(x,t)$ be the profile for the free surface of an incompressible and irrotational deep fluid in the $2\pi$-periodic domain $\mathbb{T}$ and in time $t \in \mathbb{R}$. For a proper definition of the free surface, we add the zero-mean constraint $\oint \eta(x,t) dx = 0$, which is invariant in the time evolution of Euler's equations. 

Let $\varphi(x,y,t)$ be the velocity potential, which satisfies the Laplace equation in the time-dependent spatial domain 
$$
\mathcal{D}(t) := \left\{ (x,y) \in \R^2 : \quad x \in \mathbb{T}, \quad -\infty < y \leq \eta(x,t) \right\}
$$
subject to the periodic boundary conditions on $\mathbb{T}$ and the decay condition as $y \to -\infty$. The Euler's equations are completed by two additional (kinematic and dynamic) conditions 
at the free surface $y = \eta(x,t)$:
\begin{align}
\label{euler}
\left.  \begin{array}{r} \displaystyle \eta_t + \varphi_x \eta_x - \varphi_y = 0, \\ \displaystyle
\varphi_t + \frac{1}{2} (\varphi_x)^2 + \frac{1}{2} (\varphi_y)^2 + \eta = 0, \end{array} \right\} \qquad \mbox{\rm at} \;\; y = \eta(x,t), 
\end{align}  
where the gravity constant $g$ is set to unity for convenience. 

Consider now a holomorphic function $z(u,t) = \xi(u,t) + i\eta(u,t)$, which realizes a conformal mapping of the vertical strip in the lower complex half-plane $u\in \mathbb{T} \times i(-\infty,0]$ to the fluid domain $z(\cdot,t) \in\mathcal{D}(t)$ beneath the free surface. The top boundary $\mbox{Im}\,u = 0$ gives the free surface in parametric form $x = \xi(u,t)$ and $y = \eta(u,t)$ written in variables $u \in \mathbb{T}$ and $t \in \mathbb{R}$ with $\xi = u - \mathcal{H} \eta$, where $\mathcal{H}$ is the periodic Hilbert transform in $L^2(\mathbb{T})$ normalized by the Fourier symbol 
$$
\hat{\mathcal{H}}_n = \left\{ \begin{array}{ll} 
i \; {\rm sgn}(n), \quad & n \in \mathbb{Z} \backslash \{0\}, \\
0, \quad & n = 0. \end{array} \right.
$$  
We also define a positive self-adjoint operator $K = - \mathcal{H} \partial_u$ in $L^2(\mathbb{T})$ with 
the domain $H^1_{\rm per}(\mathbb{T})$ and the Fourier symbol 
$$
\hat{K}_n = |n|, \quad n \in \mathbb{Z}.
$$
It follows from $\xi = u - \mathcal{H} \eta$ that 
$$
\xi_u = 1 + K \eta \quad \mbox{\rm and} \quad \xi_t = -\mathcal{H}\eta_t.
$$
The mean value of $\eta$ in variable $u \in \mathbb{T}$ might be a function of time $t \in \mathbb{R}$ but plays no role in the equations of motion.

By using the constrained Lagrange minimization, see \cite{DyachenkoEtAl1996}
and \cite[Appendix A]{locke2024smooth}, the system of Euler's equations in physical coordinates (\ref{euler}) can be rewritten as the following system of pseudo-differential equations for velocity potential $\psi$ and the free surface $\eta$ defined at the top boundary $\mbox{Im}\,u = 0$:
\begin{equation}
\label{time-Bab-eq}
\left\{ \begin{array}{l}
\eta_t (1 + K \eta)  + \eta_u \mathcal{H} \eta_t + \mathcal{H} \psi_u = 0, \\
\psi_t \eta_u - \psi_u \eta_t + \eta \eta_u + \mathcal{H} \left( (1 + K \eta) \psi_t + 
\psi_u \mathcal{H} \eta_t + \eta  (1 + K \eta) \right) = 0.
\end{array}
\right. 
\end{equation}
The system (\ref{time-Bab-eq}) is the starting point of our work. In the rest of this section, we review the conserved quantities, the traveling wave formulation, the existence problem for traveling waves, and the linear stability problem for traveling waves with respect to co-periodic perturbations. 

\subsection{Conserved quantities}

Taking the mean value of the two equations in system (\ref{time-Bab-eq}) yields the existence of the following two conserved quantities:
\begin{align}
\label{conserved-mass}
M(\eta) &= \oint \eta (1+K \eta) du, \\
\label{conserved-momentum}
P(\psi,\eta) &= -\oint \psi \eta_u du. 
\end{align}
Due to the zero-mean constraint $\oint \eta(x,t) dx = 0$ on the surface elevation $\eta$ and the chain rule $dx = (1 + K \eta) du$, we get the constraint $M(\eta) = 0$, or explicitly
\begin{equation}
\label{zero-mean}
\oint \eta (1+K \eta) du = 0.
\end{equation}
With the constraint $M(\eta) = 0$, another conserved quantity can be derived from the second equation in system (\ref{time-Bab-eq}):
\begin{equation} 
Q(\psi,\eta) = \oint \psi (1 + K \eta) du,
\label{conserved-mass-2}
\end{equation}
which corresponds to the conserved mean value of the potential $\psi$ on the surface in physical variable $x \in \mathbb{T}$ due to the chain rule $dx = (1+K \eta) du$.

The conserved quantities (\ref{conserved-mass}), (\ref{conserved-momentum}), and (\ref{conserved-mass-2}) follow from the general study of symmetries and conserved quantities for Euler's equations in physical variable $x \in \mathbb{R}$ in \cite{benjamin1982hamiltonian}, where $M(\eta)$, $P(\psi,\eta)$, and $Q(\psi,\eta)$ are referred to as {\em mass}, {\em the horizontal momentum}, and {\em the vertical momentum}. The same list of conserved quantities in the conformal variable $u \in \mathbb{T}$ can also be found in \cite{DyachenkoEtAl1996}. The two components of momentum can be expressed in complex form
\begin{align}
Q - iP = \oint \psi z_u \,du, 
\end{align}
and yield the complex conserved momentum, where $z_u = \xi_u + i \eta_u = 1 + K \eta + i \eta_u$.

To derive the energy conservation, we use the zero-mean constraint (\ref{zero-mean}) and the conservation of $Q(\psi,\eta)$ in (\ref{conserved-mass-2}). Applying $\mathcal{H}$ to the second equation of system (\ref{time-Bab-eq}) with $\mathcal{H}^2 = - {\rm Id}$ in the space of $2\pi$-periodic function with zero mean, we obtain 
\begin{equation}
\label{time-Bab-eq-add}
\psi_t (1 + K \eta) + 
\psi_u \mathcal{H} \eta_t + \eta (1 + K \eta) - \mathcal{H} (\psi_t \eta_u - \psi_u \eta_t + \eta \eta_u) = 0.
\end{equation}
Multiplying the first equation of system (\ref{time-Bab-eq}) by $\psi_t$ and equation (\ref{time-Bab-eq-add}) by $\eta_t$, integrating over the period of $\mathbb{T}$, and subtracting one equation from another, we integrate by parts and obtain the conserved {\em energy} (Hamiltonian) in the form:
\begin{equation}
\label{conserved-Ham}
H(\psi,\eta) = \frac{1}{2}\oint \left( \psi K \psi + \eta^2 (1 + K \eta) \right) du. 
\end{equation}
The energy $H(\psi,\eta)$ is the main quantity in the stability analysis of the traveling waves.

\subsection{Formulation of equations of motion in the traveling frame}

Let us write the first equation of system (\ref{time-Bab-eq}) 
and equation (\ref{time-Bab-eq-add}) in the reference frame moving with the wave speed $c$:
\begin{equation*}
\left\{ \begin{array}{l}
\eta_t (1 + K \eta) + \eta_u \mathcal{H} \eta_t + \mathcal{H} \psi_u - c \eta_u = 0, \\
\psi_t (1 + K \eta) + \psi_u \mathcal{H} \eta_t + \eta (1 + K \eta)- c \psi_u - \mathcal{H} (\psi_t \eta_u - \psi_u \eta_t + \eta \eta_u) = 0,
\end{array}
\right. 
\end{equation*}
where $u$ now stands for $u - ct$. Let us introduce the following change of variables by 
\begin{equation}
\label{decomposition}
\psi = -c \mathcal{H} \eta + \zeta,
\end{equation}
after which the equations of motion yield,
\begin{equation*}
\left\{ \begin{array}{l}
\eta_t (1 + K \eta) + \eta_u \mathcal{H} \eta_t + \mathcal{H} \zeta_u  = 0, \\
\zeta_t (1 + K \eta) + \zeta_u \mathcal{H} \eta_t + \eta  (1 + K \eta) - c \zeta_u - c \mathcal{H} \eta_t - c^2 K \eta \\
- \mathcal{H} \left(
\zeta_t \eta_u - \zeta_u \eta_t + \eta \eta_u 
- c  \eta_u \mathcal{H} \eta_t - c \eta_t K \eta \right)  = 0.
\end{array}
\right. 
\end{equation*}
Substituting $\zeta_u = \mathcal{H} (\eta_t (1 + K \eta) + \eta_u \mathcal{H} \eta_t )$ from the first equation to the second equation transforms
the system of evolution equations to the final form:
\begin{align}
\left\{ \begin{array}{l}
\eta_t (1 + K \eta) + \eta_u \mathcal{H} \eta_t = K\zeta, \\
\zeta_t (1 + K \eta) + \zeta_u \mathcal{H} \eta_t - \mathcal{H} (\zeta_t \eta_u - \zeta_u \eta_t ) - 2 c \mathcal{H} \eta_t  = \left(c^2 K - 1\right)\eta  - \eta K \eta - \frac{1}{2} K \eta^2.
\end{array}
\right. 
\label{full-single-eq}
\end{align}
We are now ready to set up the existence and linear stability problems for traveling waves. 

\subsection{Existence of traveling waves}

Traveling waves correspond to the reduction $\zeta = 0$ for the time-independent solutions of system (\ref{full-single-eq}). This gives the scalar pseudo-differential Babenko's equation \cite{babenko1987some} for the profile $\eta = \eta(u)$:
\begin{equation}
\label{trav-Bab-eq}
(c^2 K - 1) \eta  = \frac{1}{2} K \eta^2 + \eta K \eta.
\end{equation}
This equation can be obtained as the Euler--Lagrange equation for the action functional
\begin{equation}
\label{aug-energy}
\Lambda_c(\eta) := \frac{1}{2} \langle (c^2 K - 1) \eta, \eta \rangle - 
\frac{1}{2} \langle K \eta^2, \eta \rangle, \quad \eta \in H^1_{\rm per}(\mathbb{T}),
\end{equation}
where $\langle f, g \rangle := \frac{1}{2\pi} \oint \bar{f}(u) g(u) du$ is a standard normalized inner product in $L^2(\mathbb{T})$. We make the following assumption of existence of traveling waves, based on numerical 
results \cite{korotkevich2022superharmonic} and the small-amplitude expansions \cite{DyachenkoEtAl1996}.

\begin{assumption}
	\label{ass-existence}
There exists a family of smooth traveling waves with the even profile $\eta \in C^{\infty}_{\rm per}(\mathbb{T})$ satisfying the Babenko equation (\ref{trav-Bab-eq}) for $c \in (1,c_*)$ with some $c_* > 1$ such that 
$$
\| \eta \|_{H^1_{\rm per}} \to 0 \quad \mbox{\rm as} \;\; c \to 1.
$$
The point $c = 1$ is the bifurcation point of the $2\pi$-periodic solutions with the even single-lobe profile $\eta \sim a \cos(u) + \mathcal{O}(a^2)$ from the zero solution of the Babenko equation (\ref{trav-Bab-eq}).
\end{assumption}

\begin{remark}
    As suggested in Figure \ref{fig:hamiltonian}, the profile $\eta \in C^{\infty}_{\rm per}(\mathbb{T})$ can be more efficiently parameterized by steepness $s$ rather than speed $c$ and the dependence of speed $c$ versus steepness $s$ becomes oscillatory towards the limiting wave with the peaked profile. The details of this dependence are not important for the stability analysis as long as the zero eigenvalue bifurcation (at the extremal point of energy) is different from the extremal point of speed, see Assumption \ref{ass-kernel}.
\end{remark}

\subsection{Linear stability of traveling waves}

Expanding system (\ref{full-single-eq}) for $(\eta,\zeta)$ near the traveling wave with the profile $(\eta,0)$ and truncating the system at the linear terms with respect to the co-periodic perturbation $(v,w)$, we obtain the linearized equations of motion (also derived in \cite{dyachenko_semenova2022}):
\begin{equation}
\left\{ \begin{array}{ll}
\mathcal{M} v_t \qquad \qquad & = K w, \\
\mathcal{M}^* w_t - 2 c \mathcal{H} v_t & = \mathcal{L} v,
\end{array}
\right. 
\label{lin-Bab-eq}
\end{equation}
where 
\begin{align*}
\mathcal{M} := 1 + K \eta + \eta' \mathcal{H} \quad\mbox{and}\quad
\mathcal{M}^* := 1 + K \eta - \mathcal{H} (\eta' \; \cdot ), 
\end{align*}
and
\begin{align}
\mathcal{L} := c^2 K - (1 + K \eta) - \eta K - K(\eta \; \cdot ). \label{linBop}
\end{align}
We note that $\mathcal{M}^*$ is the adjoint operator to a bounded operator $\mathcal{M}$ in $L^2(\mathbb{T})$ with respect to $\langle \cdot, \cdot \rangle$ and that $\mathcal{L}$ is a self-adjoint unbounded operator in $L^2(\mathbb{T})$ with ${\rm Dom}(\mathcal{L}) = H^1_{\rm per}(\mathbb{T})$. Furthermore, 
$\mathcal{L}$ is the linearized operator of the Babenko equation (\ref{trav-Bab-eq}). 
Also recall that $K$ is a self-adjoint unbounded operator in $L^2(\mathbb{T})$ with ${\rm Dom}(K) = H^1_{\rm per}(\mathbb{T})$. It is clear from the Fourier series that $K 1 = 0$, and ${\rm Ker}(K) = {\rm span}(1)$.

\begin{remark}
It follows from the translational symmetry of the Babenko equation (\ref{trav-Bab-eq}) that 
$\mathcal{L} \eta' = 0$ with $\eta' \in H^1_{\rm per}(\mathbb{T})$ if $\eta \in C^{\infty}_{\rm per}(\mathbb{T})$ is smooth. We also note that 
\begin{equation}
\label{L-on-1}
\mathcal{L} 1 = -\left(1 + 2K\eta\right),
\end{equation} 
which is useful in our computations.
\end{remark}

Separating variables in the linearized system (\ref{lin-Bab-eq}) yields the spectral stability problem with respect to co-periodic perturbations, 
\begin{equation}
\left\{ \begin{array}{ll}
Kw &= \lambda \mathcal{M} v, \\
\mathcal{L} v &= \lambda (\mathcal{M}^* w - 2 c \mathcal{H} v),
\end{array}
\right. 
\label{spec-Bab-eq-aux}
\end{equation}
where $(v,w) \in H^1_{\rm per}(\mathbb{T}) \times H^1_{\rm per}(\mathbb{T})$ is an eigenfunction and $\lambda \in \mathbb{C}$ is an eigenvalue. Since $K$ and $\mathcal{L}$ are unbounded operators and $\mathbb{T}$ is compact, the spectrum of the spectral problem (\ref{spec-Bab-eq-aux}) consists of eigenvalues of finite algebraic multiplicity. 

\begin{remark}
There exist two linearly independent eigenfunctions in the kernel of the spectral stability problem (\ref{spec-Bab-eq-aux}) due to the following two symmetries of the underlying physical system. A spatial translation of the Stokes wave results in another solution of the Babenko equation (\ref{trav-Bab-eq}), and is associated with a one-dimensional subspace spanned by the eigenfunction $(v,w) = (\eta',0)$. Similarly, the fluid potential admits gauge transformation $\psi(u,t) \to \psi(u,t) + \psi_0(t)$ for any function $\psi_0(t)$. This property is associated with a one-dimensional subspace spanned by the eigenfunction $(v,w)=(0,1)$. 
\end{remark}

\section{Criterion for instability bifurcation}
\label{sec-3}

We rewrite the spectral stability problem (\ref{spec-Bab-eq-aux}) in the matrix form 
\begin{equation}
\label{spec-Bab-eq}
\left( \begin{matrix} 0 & K \\ \mathcal{L} & 0 \end{matrix} \right) 
\left( \begin{matrix} v \\ w \end{matrix} \right) = \lambda 
\left( \begin{matrix} \mathcal{M} & 0 \\ -2 c \mathcal{H} & \mathcal{M}^* \end{matrix} \right) 
\left( \begin{matrix} v \\ w \end{matrix} \right),
\end{equation}
which is rewritten as the generalized eigenvalue problem of the form $A \vec{x} = \lambda B \vec{x}$ 
with 
\begin{align*}
A & : H^1_{\rm per}(\mathbb{T}) \times H^1_{\rm per}(\mathbb{T}) \to L^2(\mathbb{T}) \times L^2(\mathbb{T}), \\ 
B &: L^2(\mathbb{T}) \times L^2(\mathbb{T}) \to L^2(\mathbb{T}) \times L^2(\mathbb{T}), 
\end{align*}
given by 
$$
A = \left( \begin{matrix} 0 & K \\ \mathcal{L} & 0 \end{matrix} \right), \qquad 
B = \left( \begin{matrix} \mathcal{M} & 0 \\ -2 c \mathcal{H} & \mathcal{M}^* \end{matrix} \right),
$$
and $\vec{x} = (v,w) \in H^1_{\rm per}(\mathbb{T}) \times H^1_{\rm per}(\mathbb{T})$.
The geometric multiplicity of $\lambda = 0$ is defined by the dimension of ${\rm Ker}(A)$. The algebraic multiplicity of $\lambda = 0$ is defined by the length of the Jordan chain of generalized eigenvectors 
\begin{align*}
    A \vec{x}_0 &= 0, \\
    A \vec{x}_1 &= B \vec{x}_0, \\
    A \vec{x}_2 &= B \vec{x}_1, 
\end{align*}
with $\vec{x}_0, \vec{x}_1, \vec{x}_2, \dots \in H^1_{\rm per}(\mathbb{T}) \times H^1_{\rm per}(\mathbb{T})$. 
In what follows, we compute the Jordan chain for the particular operators $A$ and $B$ in (\ref{spec-Bab-eq}).

\begin{remark}
    The bounded operator $M : L^2(\mathbb{T}) \to L^2(\mathbb{T})$ is invertible with the explicit formula for the inverse operator, see \cite[Eq. (13)]{dyachenko_semenova2022}. Hence, the bounded operator 
    $B : L^2(\mathbb{T}) \times L^2(\mathbb{T}) \to L^2(\mathbb{T}) \times L^2(\mathbb{T})$ is also invertible so that the generalized eigenvalue problem $A \vec{x} = \lambda B \vec{x}$ can be rewritten as the linear eigenvalue problem $B^{-1} A \vec{x} = \lambda \vec{x}$.
\end{remark}

Since $K 1 = 0$ and $\mathcal{L} \eta' = 0$, the 
null space of the unbounded operator 
$A : H^1_{\rm per}(\mathbb{T}) \times H^1_{\rm per}(\mathbb{T}) \to L^2(\mathbb{T}) \times L^2(\mathbb{T})$
is at least two-dimensional with 
\begin{equation}
\label{kernel}
\left( \begin{array}{c} v \\ w \end{array} \right) = a_1 \left( \begin{array}{c} \eta' \\ 0 \end{array} \right) + a_2 \left( \begin{array}{c} 0 \\ 1 \end{array} \right),
\end{equation}
where $(a_1,a_2) \in \mathbb{R}^2$. Due to the Hamiltonian symmetry, 
the generalized null space of the spectral stability problem (\ref{spec-Bab-eq}) is at least four-dimensional with at least two generalized eigenfunctions, see (\ref{first-chain}) below.

\begin{definition}
We say that the periodic wave with the profile $\eta \in C^{\infty}_{\rm per}(\mathbb{T})$ is at the stability threshold if the generalized null space of the spectral stability problem (\ref{spec-Bab-eq}) has algebraic multiplicity exceeding four. 
\label{def-bifurcation}
\end{definition}

There are two possibility to hit the stability threshold of Definition \ref{def-bifurcation}: either the null space of $A$ becomes at least three-dimensional or the null space of $A$ remains two-dimensional but the generalized null space of $B^{-1} A$ becomes six-dimensional. Since ${\rm Ker}(K) = {\rm span}(1)$, the first possibility could only be realized if $\mathcal{L}$ has a double zero eigenvalue, see \cite{dyachenko_semenova2022,deconinck2022instability}. This corresponds to the fold point in the dependence of speed $c$ versus steepness $s$, see the red curve in Figure \ref{fig:hamiltonian}, since the family of 
solutions of the Babenko equation (\ref{trav-Bab-eq}) fails to continue in $c$ at the extremal values 
of the dependence of $c$ versus steepness $s$. Therefore,
we eliminate the first possibility according to the following assumption and restrict our attention to the second possibility.

\begin{assumption} 
	\label{ass-kernel}
${\rm Ker}(\mathcal{L}) = {\rm span}(\eta')$, that is, the value of $c$ is not a fold point.
\end{assumption}

Due to Assumption \ref{ass-kernel}, the mapping $c \mapsto \eta \in C^{\infty}_{\rm per}(\mathbb{T})$ is smooth so that we can differentiate the Babenko equation (\ref{trav-Bab-eq}) in $c$ and obtain 
\begin{equation}
\label{der-speed}
\mathcal{L} \partial_c \eta + 2 c K \eta = 0, \quad \Rightarrow \quad \partial_c \eta = -2 c \mathcal{L}^{-1} K \eta,
\end{equation}
where $\partial_c \eta \in C^{\infty}_{\rm per}(\mathbb{T})$ and $\mathcal{L}^{-1}$ is uniquely defined on the subspace of even functions in $L^2(\mathbb{T})$ since ${\rm Ker}(\mathcal{L}) = {\rm span}(\eta')$ is spanned by the odd function, see Assumption \ref{ass-existence}. Related to the profile $\eta \in C^{\infty}_{\rm per}(\mathbb{T})$ of the traveling wave, we define the wave action by $\mathcal{E}(c) := \Lambda_c(\eta)$, where $\Lambda_c(\eta)$ is given by (\ref{aug-energy}) and the wave momentum and energy by 
\begin{equation}
\label{wave-momentum}
\mathcal{P}(c) := P(\psi  = -c \mathcal{H} \eta,\eta) = c \langle K \eta, \eta \rangle
\end{equation}
and 
\begin{equation}
\label{wave-energy}
\mathcal{H}(c) := H(\psi= -c \mathcal{H} \eta, \eta) = \frac{c^2}{2} \langle K \eta, \eta \rangle + \frac{1}{2}\langle \eta^2, (1+K\eta) \rangle,
\end{equation}
where $P(\psi,\eta)$ and $H(\psi,\eta)$ are given by (\ref{conserved-momentum}) and (\ref{conserved-Ham}). 
The following theorem presents the main result on the criterion for instability bifurcation. 

\begin{theorem}
	\label{theorem-crit}
Under Assumptions \ref{ass-existence} and \ref{ass-kernel}, the generalized null space of the spectral problem (\ref{spec-Bab-eq}) is at least six-dimensional if and only if $\mathcal{P}'(c) = 0$ or, equivalently, $\mathcal{H}'(c) = 0$.
\end{theorem}

\begin{proof}
	Due to Assumption \ref{ass-kernel}, the null space of the spectral problem (\ref{spec-Bab-eq}) is spanned by (\ref{kernel}). The first element of the Jordan chain is defined by the periodic solutions 
	$(v_1,w_1) \in H^1_{\rm per}(\mathbb{T}) \times H^1_{\rm per}(\mathbb{T})$ of the linear inhomogeneous equations:
	\begin{equation}
	\left\{ \begin{array}{ll}
	K w_1 &=  a_1 \mathcal{M} \eta', \\
	\mathcal{L} v_1 &= - 2 c a_1 \mathcal{H} \eta' + a_2 \mathcal{M}^* 1.
	\end{array}
	\right. 
	\label{gen-Bab-eq}
	\end{equation}	
	Since $\mathcal{M} \eta' = \eta'$ and $\langle 1, \eta' \rangle = 0$, 
	there exists $w_1 \in {\rm Dom}(K)$ from the first equation in the system (\ref{gen-Bab-eq}). For unique definition of $w_1$, we take projection of $w_1$ to $1$ to be zero, after which we get $w_1  = a_1 \mathcal{H} \eta$. By Assumption \ref{ass-existence}, $\eta$ is even, which implies that 
    $\mathcal{H} \eta$ and $\eta'$ are odd. Hence, $\eta'$ has opposite parity compared to $\mathcal{H} \eta'$ and $\mathcal{M}^* 1 = 1 + 2 K \eta$ so that there exists $v_1 \in {\rm Dom}(\mathcal{L})$ from the second equation in the system (\ref{gen-Bab-eq}). For unique definition of $v_1$, we take projection of $v_1$ to $\eta'$ to be zero. By using (\ref{L-on-1}) and (\ref{der-speed}), we get the explicit solutions 
	\begin{align}
	\label{first-chain}
	v_1 = -a_1 \partial_c \eta - a_2, \quad w_1 = a_1 \mathcal{H}\eta,
	\end{align}
where $\partial_c \eta$ is even and $\mathcal{H} \eta$ is odd. 

The second element of the Jordan chain is defined by the periodic solutions 
$(v_2,w_2) \in H^1_{\rm per}(\mathbb{T}) \times H^1_{\rm per}(\mathbb{T})$ of
the linear inhomogeneous equations:
	\begin{equation}
\left\{ \begin{array}{ll}
K w_2 &=  \mathcal{M} v_1  , \\
\mathcal{L} v_2 &= - 2 c \mathcal{H} v_1 + \mathcal{M}^* w_1
\end{array}
\right. 
\label{gen-Bab-eq-2-used}
\end{equation}
which is written explicitly as 	
	\begin{equation}
	\left\{ \begin{array}{ll}
	K w_2 &=  - a_1 \mathcal{M} \partial_c \eta - a_2 \mathcal{M} 1, \\
	\mathcal{L} v_2 &= 2 c a_1 \mathcal{H}  \partial_c \eta + a_1 \mathcal{M}^* \mathcal{H} \eta.
	\end{array}
	\right. 
	\label{gen-Bab-eq-2}
	\end{equation}
Since 
\begin{align*}
\langle 1, \mathcal{M} f \rangle &= \langle \mathcal{M}^* 1, f \rangle = \langle (1+2K \eta), f \rangle, \\
\langle \eta', \mathcal{M}^* f \rangle &= \langle \mathcal{M} \eta', f \rangle = \langle \eta', f \rangle, \\
\langle \eta', \mathcal{H} f \rangle &= -\langle \mathcal{H} \eta', f \rangle = \langle K \eta, f \rangle, 
\end{align*}
Fredholm theorem implies that there exist periodic solutions of the linear inhomogeneous system (\ref{gen-Bab-eq-2}) if and only if the following linear homogeneous system on $(a_1,a_2)$ admits a nonzero solution:
\begin{equation}
\left( \begin{matrix} - \langle (1+2K \eta), \partial_c \eta \rangle & 
-\langle (1+2K \eta), 1 \rangle \\
\langle \eta', \mathcal{H} \eta \rangle + 2c \langle K\eta, \partial_c \eta \rangle & 0
\end{matrix}  
\right) \left( \begin{matrix} a_1 \\ a_2 \end{matrix} \right) = \left( \begin{matrix} 0 \\ 0 \end{matrix} \right).
\label{lin-alg}
\end{equation} 
Taking derivative of the constraint (\ref{zero-mean}) with respect to $c$ yields 
\begin{equation*}
\langle (1+2K\eta), \partial_c \eta \rangle = 0.
\end{equation*}
On the other hand, since $K$ is self-adjoint and $K 1 = 0$, we have 
\begin{align*}
\langle  1,  K \eta \rangle = \langle K1, \eta \rangle = 0.
\end{align*}
Hence, the linear system (\ref{lin-alg}) can be rewritten in the equivalent form as
\begin{equation*}
\left( \begin{matrix} 0 & -1 \\
\langle \eta, K\eta \rangle + 2c \langle K\eta, \partial_c \eta \rangle & 
0
\end{matrix}  
\right) \left( \begin{matrix} a_1 \\ a_2 \end{matrix} \right) = \left( \begin{matrix} 0 \\ 0 \end{matrix} \right).
\end{equation*} 
Thus, $a_2 = 0$, whereas $a_1 \neq 0$ if and only if 
\begin{equation}
\label{det-eq}
\mathcal{D}(c) := \langle K \eta, \eta \rangle + 2c \langle K\eta, \partial_c \eta \rangle 
= \frac{d}{dc} c  \langle K \eta, \eta \rangle = \mathcal{P}'(c) = 0.
\end{equation}
In the case of $\mathcal{P}'(c) = 0$, the second element of the Jordan chain is represented by the only periodic solution of the linear system (\ref{gen-Bab-eq-2}) in the form $(v_2,w_2) = a_1 (\tilde{v}_2,\tilde{w}_2)$, 
where $(\tilde{v}_2,\tilde{w}_2) \in  H^1_{\rm per}(\mathbb{T}) \times H^1_{\rm per}(\mathbb{T})$ are uniquely defined from solutions of the linear inhomogeneous equations
\begin{equation}
\left\{ \begin{array}{ll}
K \tilde{w}_2 &=  -\mathcal{M} \partial_c \eta, \\
\mathcal{L} \tilde{v}_2 &= 2c \mathcal{H} \partial_c \eta + \mathcal{M}^* \mathcal{H} \eta - \frac{\mathcal{P}'(c)}{\| \eta' \|^2} \eta',
\end{array}
\right. 
\label{gen-Bab-eq-2-red}
\end{equation}
subject to the orthogonality conditions 
\begin{equation}
    \label{orth-2-red}
\langle 1, \tilde{w}_2 \rangle = 0 \quad \mbox{\rm and} \quad \langle \eta', \tilde{v}_2 \rangle = 0.
\end{equation}

\begin{remark}
    We have added the orthogonal projection to the second equation of system (\ref{gen-Bab-eq-2-red}) even though $\mathcal{P}'(c) = 0$. This is useful for numerical approximations as well as for the derivation of the normal form in Theorem \ref{theorem-bif}.
\end{remark}

To prove that the generalized null space of the spectral problem (\ref{spec-Bab-eq}) is at least six-dimensional, we consider the third element of the Jordan chain defined by the periodic solutions $(v_3,w_3) \in  H^1_{\rm per}(\mathbb{T}) \times H^1_{\rm per}(\mathbb{T})$ of the linear inhomogeneous equations 
\begin{equation}
\left\{ \begin{array}{ll}
K w_3 &=  a_1 \mathcal{M} \tilde{v}_2, \\
\mathcal{L} v_3 &= - 2 c a_1 \mathcal{H} \tilde{v}_2 + a_1 \mathcal{M}^* \tilde{w}_2.
\end{array}
\right. 
\label{gen-Bab-eq-3}
\end{equation}	
Since $\eta$ is even and operators $K$, $\mathcal{M}$ and $\mathcal{L}$ are parity preserving, we obtain from (\ref{gen-Bab-eq-2-red}) and the orthogonality conditions (\ref{orth-2-red}) that $\tilde{w}_2$ is even and $\tilde{v}_2$ is odd. Hence, odd $\mathcal{M} \tilde{v}_2$ is orthogonal to even $1$ and even $- 2 c \mathcal{H} \tilde{v}_2 + \mathcal{M}^* \tilde{w}_2$ is orthogonal to odd $\eta'$. By Fredholm's theorem, the third element of the Jordan chain is represented by the only periodic solution of the linear system (\ref{gen-Bab-eq-3}) in the form $(v_3,w_3) = a_1 (\tilde{v}_3,\tilde{w}_3)$, 
where $(\tilde{v}_3,\tilde{w}_3) \in  H^1_{\rm per}(\mathbb{T}) \times H^1_{\rm per}(\mathbb{T})$ are uniquely defined from solutions of the linear inhomogeneous equations
\begin{equation}
\left\{ \begin{array}{ll}
K \tilde{w}_3 &=  \mathcal{M} \tilde{v}_2, \\
\mathcal{L} \tilde{v}_3 &= - 2 c \mathcal{H} \tilde{v}_2 + \mathcal{M}^* \tilde{w}_2.
\end{array}
\right. 
\label{gen-Bab-eq-3-red}
\end{equation}
subject to the orthogonality conditions 
\begin{equation}
    \label{orth-3-red}
\langle 1, \tilde{w}_3 \rangle = 0 \quad \mbox{\rm and} \quad \langle \eta', \tilde{v}_3 \rangle = 0.
\end{equation}
From the same parity argument and the orthogonality conditions (\ref{orth-3-red}), we conclude that $\tilde{w}_3 \in H^1_{\rm per}(\mathbb{T})$ is odd and $\tilde{v}_3 \in H^1_{\rm per}(\mathbb{T})$ is even. Thus, the generalized null space of the spectral problem (\ref{spec-Bab-eq}) is at least six-dimensional if and only if $\mathcal{P}'(c) = 0$. 

It remains to show that the critical points of $\mathcal{P}(c)$ coincide with the critical points of the energy  $\mathcal{H}(c)$. Differentiating the action $\mathcal{E}(c) := \Lambda_c(\eta)$ given by (\ref{aug-energy}) in $c$ yields 
\begin{equation}
    \label{aug-energy-der}
	\mathcal{E}'(c) = c \langle K \eta, \eta \rangle + \langle (c^2 K \eta - \eta - \frac{1}{2} K \eta^2 - \eta K \eta), \partial_c \eta \rangle = \mathcal{P}(c), 
\end{equation}
where the quantity in the brackets vanishes due to the Babenko equation (\ref{trav-Bab-eq}). 
Since $\mathcal{E}(c) = c \mathcal{P}(c) - \mathcal{H}(c)$, we obtain 
\begin{align*}
\mathcal{E}'(c) &= \mathcal{P}(c) + c \mathcal{P}'(c) - \mathcal{H}'(c),
\end{align*}
which yields $\mathcal{H}'(c)= c \mathcal{P}'(c)$.
\end{proof}

\begin{remark}
	The two orthogonality conditions used in the proof of Theorem \ref{theorem-crit} can be stated for every eigenfunction $(v,w) \in H^1_{\rm per}(\mathbb{T}) \times H^1_{\rm per}(\mathbb{T})$ of the spectral problem (\ref{spec-Bab-eq}) with $\lambda \neq 0$. Indeed, the two Fredholm constraints 
	\begin{align*}
	0 &= \langle 1, K w \rangle = \lambda \langle 1, \mathcal{M}v \rangle = \lambda \langle \mathcal{M}^* 1, v \rangle, \\
	0 &= \langle \eta', \mathcal{L} v \rangle = \lambda \left( \langle \eta', \mathcal{M}^* w \rangle - 2 c \langle \eta', \mathcal{H} v \rangle \right) 
	= \lambda \left( \langle \mathcal{M} \eta', w \rangle + 2 c \langle \mathcal{H} \eta', v \rangle \right)
	\end{align*}
imply
\begin{equation}
\label{orth-conditions}
	\langle (1+ 2K \eta), v \rangle = 0, \qquad 
	\langle \eta', w \rangle - 2 c \langle K \eta, v \rangle = 0.
\end{equation}
The first orthogonality condition in (\ref{orth-conditions}) 
is a linearization of the constraint (\ref{zero-mean}). 
The second orthogonality condition in (\ref{orth-conditions}) 
is a linearization of the momentum conservation $P(\psi,\eta)$ with the decomposition (\ref{decomposition}):
$$
P(\psi,\eta) = c \langle K \eta, \eta \rangle - \langle \eta_u, \zeta \rangle,
$$
since $(v,w)$ is the perturbation of the traveling wave with the profile $(\eta,0)$ in variables $(\eta,\zeta)$.
\end{remark}

\begin{remark}
    It follows from the proof of Theorem \ref{theorem-crit} that the Jordan canonical form 
    of $B^{-1} A$ at $\lambda = 0$ if $\mathcal{P}'(c) \neq 0$ 
    is 
    $$
    \left( 
    \begin{matrix}
    0 & 1 & 0 & 0\\ 0 & 0 & 0 & 0 \\ 0 & 0 & 0 & 1 \\ 0 & 0 & 0 & 0
    \end{matrix}
    \right). 
    $$
    If $\mathcal{P}'(c) = 0$ and $\mathcal{B} \neq 0$, see (\ref{def-D}) below, the Jordan canonical form     of $B^{-1} A$ for $\lambda = 0$ is 
        $$
    \left( 
    \begin{matrix}
    0 & 1 & 0 & 0 & 0 & 0\\ 0 & 0 & 1 & 0 & 0 & 0 \\ 0 & 0 & 0 & 1 & 0 & 0 \\ 0 & 0 & 0 & 0 & 0 & 0 \\
   0 & 0 & 0 & 0 & 0 & 1 \\    0 & 0 & 0 & 0 & 0 & 0
    \end{matrix}
    \right).
    $$
\end{remark}

\section{Normal form for unstable eigenvalues}
\label{sec-4}

We derive the normal form for the splitting of the multiple zero eigenvalue 
of the spectral problem (\ref{spec-Bab-eq}) for the values of $c$ close to a critical point of $\mathcal{P}(c)$ in Theorem \ref{theorem-crit}. The following theorem gives the main result.

\begin{theorem}
	\label{theorem-bif}
Under Assumptions \ref{ass-existence} and \ref{ass-kernel}, let $c_0 > 1$ be the extremal point of $\mathcal{P}(c)$ such that $\mathcal{P}'(c_0) = 0$ and $\mathcal{P}''(c_0) \neq 0$. Then, there is $\epsilon_0 > 0$ such that for every $c \in (c_0,c_0+\epsilon_0)$, the spectral stability problem (\ref{spec-Bab-eq}) admits two (small) real eigenvalues $\pm \lambda(c)$ with $\lambda(c) > 0$ near $0$ 
if $\mathcal{B} \mathcal{P}''(c_0) < 0$ and two (small) purely imaginary eigenvalues 
$\pm i \omega(c)$ with $\omega(c) > 0$ near $0$ if 	$\mathcal{B} \mathcal{P}''(c_0) > 0$, where 
	\begin{equation}
	\label{def-D}
\mathcal{B} := \langle \eta', \tilde{w}_3 \rangle - 2 c \langle K \eta, \tilde{v}_3 \rangle
	\end{equation} 
	is defined from solutions of (\ref{gen-Bab-eq-3-red}) computed from solutions of (\ref{gen-Bab-eq-2-red}). The real and purely imaginary eigenvalues are exchanged to the opposite if $c \in (c_0-\epsilon_0,c_0)$.
\end{theorem}

\begin{proof}
Since $\mathcal{P}'(c_0) = 0$ and $c \in (c_0 - \epsilon_0,c_0+\epsilon_0)$ for small $\epsilon_0 > 0$, we expand $\mathcal{P}'(c) = \mathcal{P}''(c_0) (c-c_0) + \mathcal{O}((c-c_0)^2)$ and assume that $\mathcal{P}''(c_0) \neq 0$. Let $\epsilon := |c - c_0| \in (0,\epsilon_0)$ be a small parameter. Solutions to the spectral stability problem (\ref{spec-Bab-eq}) are found by Puiseux expansion for the multiple zero eigenvalue:
\begin{equation}
    \left\{ \begin{array}{l} 
    v = \eta' + \sqrt{\epsilon} v_1 + \epsilon v_2 + \epsilon \sqrt{\epsilon} v_3 + \epsilon^2 v_4 + \mathcal{O}(\epsilon^2 \sqrt{\epsilon}), \\
    w = 0 + \sqrt{\epsilon} w_1 + \epsilon w_2 + \epsilon \sqrt{\epsilon} w_3 + \epsilon^2 w_4 + \mathcal{O}(\epsilon^2 \sqrt{\epsilon}),\\
    \lambda = 0 + \sqrt{\epsilon} \lambda_1 + \epsilon \lambda_2 + \epsilon \sqrt{\epsilon} \lambda_3 + \epsilon^2 \lambda_4 + \mathcal{O}(\epsilon^2 \sqrt{\epsilon}), \end{array} \right. \label{halfseries}
\end{equation}
where all correction terms are to be found recursively. Since the admissible values of $\lambda_1$ are found at the order of $\mathcal{O}(\epsilon^4)$ and the admissible values of $\lambda_2$, $\lambda_3$, etc are found at the higher orders, we will not write any correction terms related to $\lambda_2$, $\lambda_3$, etc. They are identical to the correction terms related to $\lambda_1$.

At the order of $\mathcal{O}(\sqrt{\epsilon})$, we obtain the linear inhomogeneous system (\ref{gen-Bab-eq}) with $a_1 = \lambda_1$ and $a_2 = 0$, hence the solution is 
$$
v_1 = -\lambda_1 \partial_c \eta, \qquad w_1 = \lambda_1 \mathcal{H} \eta, 
$$
in agreement with (\ref{first-chain}).

At the order of $\mathcal{O}(\epsilon)$, we obtain the linear inhomogeneous system (\ref{gen-Bab-eq-2}) with $a_1 = \lambda_1^2$ and $a_2 = 0$. Recall that the solution of (\ref{gen-Bab-eq-2}) exists if and only if $\mathcal{P}'(c) = 0$, which is not the case if $c \neq c_0$ due to $\epsilon \neq 0$. Therefore, we represent the solution of (\ref{gen-Bab-eq-2}) in the form
$$
v_2 = \lambda_1^2 \tilde{v}_2, \qquad w_2 = \lambda_1^2 (\tilde{w}_2 + \alpha), 
$$
where $(\tilde{v}_2,\tilde{w}_2)$ is a solution of the linear inhomogeneous system (\ref{gen-Bab-eq-2-red}) uniquely defined under the orthogonality conditions (\ref{orth-2-red}) and $\alpha \in \mathbb{R}$ is a parameter to be determined from the orthogonality condition at the order of $\mathcal{O}(\epsilon^2)$ due to ${\rm Ker}(K) = {\rm span}(1)$.

\begin{remark}
    When $(\tilde{v}_2,\tilde{w}_2)$ is obtained from the system (\ref{gen-Bab-eq-2-red}), we use the additional term $-\frac{\mathcal{P}'(c)}{\| \eta' \|^2} \eta'$ in the system (\ref{gen-Bab-eq-2-red}) required by the Fredholm theorem. This term of order $\mathcal{O}(\epsilon)$ is compensated for by the term $+\frac{\mathcal{P}'(c)}{\epsilon \| \eta' \|^2} \eta'$ of order $\mathcal{O}(\epsilon^2)$, see (\ref{gen-Bab-eq-4} below. This is consistent with the assumption $\mathcal{P}''(c_0) \neq 0$, which ensures that 
    $\mathcal{P}'(c) = \mathcal{O}(\epsilon)$.
\end{remark}

At the order of $\mathcal{O}(\epsilon \sqrt{\epsilon})$, we obtain the linear inhomogeneous system, 
\begin{equation}
\left\{ \begin{array}{ll}
K w_3 &=  \lambda_1^3 \mathcal{M} \tilde{v}_2, \\
\mathcal{L} v_3 &= \lambda_1^3 (- 2 c \mathcal{H} \tilde{v}_2 + \mathcal{M}^* \tilde{w}_2 + \alpha \mathcal{M}^* 1).
\end{array}
\right. 
\label{gen-Bab-eq-3-mod}
\end{equation}	
which can be compared with (\ref{gen-Bab-eq}) and (\ref{gen-Bab-eq-3}). The solution exists in the form
$$
v_3 = \lambda_1^3 (\tilde{v}_2 - \alpha), \qquad w_3 = \lambda_1^3 \tilde{w}_3, 
$$
where $(\tilde{v}_3,\tilde{w}_3)$ is a solution of the linear inhomogeneous equation (\ref{gen-Bab-eq-3-red}) uniquely defined under orthogonality conditions (\ref{orth-3-red}).

At the order of $\mathcal{O}(\epsilon^2)$, we obtain the linear inhomogeneous system, 
\begin{equation}
\left\{ \begin{array}{ll}
K w_4 &=  \lambda_1^4 (\mathcal{M} \tilde{v}_3 - \alpha \mathcal{M} 1), \\
\mathcal{L} v_4 &= \lambda_1^4 (- 2 c \mathcal{H} \tilde{v}_3 + \mathcal{M}^* \tilde{w}_3) + \lambda_1^2 \epsilon^{-1} \frac{\mathcal{P}'(c)}{\|\eta'\|^2} \eta',
\end{array}
\right. 
\label{gen-Bab-eq-4}
\end{equation}	
where the projection term came from the order $\mathcal{O}(\epsilon)$ in the linear inhomogeneous system (\ref{gen-Bab-eq-2-red}). Since ${\rm Ker}(K) = {\rm span}(1)$, $\alpha \in \mathbb{R}$ is uniquely found from 
the existence of the solution $w_4 \in H^1_{\rm per}(\mathbb{T})$ by the Fredholm theorem:
$$
\alpha = \frac{\langle 1, \mathcal{M} \tilde{v}_3 \rangle}{\langle 1, \mathcal{M} 1 \rangle} = 
\langle (1 + 2 K \eta), \tilde{v}_3 \rangle,
$$
where we have used $\langle 1, \mathcal{M} 1 \rangle = \langle (1+2K\eta), 1 \rangle = 1$. 
Although $\alpha$ is uniquely defined from the first equation of the system (\ref{gen-Bab-eq-4}), 
it does not contribute to the second equation of the system (\ref{gen-Bab-eq-4}) and therefore does not change the normal form. Since ${\rm Ker}(\mathcal{L}) = {\rm span}(\eta')$, a solution $v_4 \in H^1_{\rm per}(\mathbb{T})$ exists if and only if 
$$
\lambda_1^4 \langle \eta', (- 2 c \mathcal{H} \tilde{v}_3 + \mathcal{M}^* \tilde{w}_3) \rangle + \lambda_1^2 \epsilon^{-1} \mathcal{P}'(c) = 0, 
$$
which can be rewritten at the leading order of $\mathcal{P}'(c) = \mathcal{P}''(c_0) (c-c_0) + \mathcal{O}(\epsilon^2)$ as 
\begin{equation}
\label{normal-form-eq}
\lambda_1^4 \mathcal{B} + \lambda_1^2  {\rm sgn}(c-c_0) \mathcal{P}''(c_0) = 0, 
\end{equation}
where $\mathcal{B}$ is given by (\ref{def-D}). A nonzero solution for $\lambda_1$ exists if $\mathcal{P}''(c_0) \neq 0$ such that for ${\rm sgn}(c-c_0) = +1$, that is, for $c \in (c_0,c_0+\epsilon_0)$, we have $\lambda_1^2 > 0$ if $\mathcal{B} \mathcal{P}''(c_0) < 0$ and $\lambda_1^2 < 0$ if $\mathcal{B} \mathcal{P}''(c_0) > 0$. 
The sign of $\lambda_1^2$ is exchanged to the opposite if ${\rm sgn}(c-c_0) = -1$, that is, if $c \in (c_0-\epsilon_0,c_0)$. This concludes the proof. 
\end{proof}

\begin{remark}
As Figure \ref{fig:hamiltonian} shows, the periodic wave with the even profile $\eta \in C^{\infty}_{\rm per}(\mathbb{T})$ is continued numerically with respect to the steepness parameter $s$. Based on the numerical observations in Figure \ref{fig:hamiltonian}, there is exactly one fold point between each extremal point of $\mathcal{P}(c)$ or, equivalently, $\mathcal{H}(c)$, so that $\frac{dc}{ds}$ alternates sign at each point, where 
$\mathcal{P}'(c_0) = 0$. Similarly, the sign of $\mathcal{P}''(c_0)$ alternates between the extremal points. Since we show in Section \ref{sec-5} that the value of $\mathcal{B}$ has the same sign for each instability bifurcation, the new pair of real (unstable) eigenvalues $\lambda$ bifurcates in the direction of increasing steepness $s$ at each extremal point of $\mathcal{P}(c)$.
\end{remark}

\section{Numerical approximations}
\label{sec-5}

We are considering here Stokes waves with large steepness, which are beyond the applicability limit of the small-amplitude expansions~\cite{levi1925determination}. Table~\ref{tab:wave_params} gives the first and second critical points of the Hamiltonian $H$ at steepness $s_1$ and $s_2$, which are also seen in Figure \ref{fig:hamiltonian}. It is then necessary to use other approximations of Stokes waves with large steepness before the stability problem can be studied. There are two challenging problems that have to be treated numerically, the first one is obtaining a solution of the Babenko equation~\eqref{trav-Bab-eq} with high accuracy, and the second one is to find the eigenvalues of the stability problem~\eqref{spec-Bab-eq}. Once the eigenvalues are found, we can compare them with the Puiseux expansion in~\eqref{halfseries} to cross-validate numerics and theory.

\begin{table}[htb!]
    \centering
    \renewcommand{\arraystretch}{1.3}
    \begin{tabular}{c|c|c|c|c}
        $s$ & $c$ & $\mathcal H$ & $\mathcal P$ & $\mathcal{B} $\\ \hline
         $0.13660354990$ & $1.0921379$ & $0.46517718146$ & $0.44729319629$ &  $11.01822$\\ 
        $0.14079654715$ & $1.0922868$ & $0.45770578965$ & $0.44045605242$ & $10.96232$\\ 
    \end{tabular}
    \caption{Parameters of Stokes waves at the first two extrema of energy $\mathcal{H}$ or, equivalently, the horizontal momentum $\mathcal{P}$ with the coefficient $\mathcal{B}$ of the normal form (\ref{def-D}). }
    \label{tab:wave_params}
\end{table}

\subsection{The Babenko equation}

We adopt the strategy from~\cite{dyachenko_semenova2022,dyachenko2022almost} to find Stokes waves numerically. The entire branch of Stokes waves is found by continuation method with respect to the speed parameter $c$. Given a Stokes wave $\eta^{(0)}(u)$ with speed $c^{(0)}$ that solves $\mathcal{S}(c^{(0)},\eta^{(0)}) = 0$, where $\mathcal{S}$ denotes the nonlinear Babenko equation~\eqref{trav-Bab-eq}, we apply the Newton's method to find a new solution $(c^{(1)},\eta^{(1)})$. The initial approximation to Stokes wave with $c^{(1)}$ is chosen to be $\eta^{(0)} + \delta \eta$ from the expansion
\begin{align*}
   0 = \mathcal{S}(c^{(1)}, \eta^{(0)} + \delta \eta) = \mathcal{S}(c^{(1)}, \eta^{(0)}) + \mathcal L(c^{(1)},\eta^{(0)})\,\delta \eta + \ldots,
\end{align*}
where the neglected terms are quadratic and higher order functions of $\delta \eta$ and $\mathcal{L}(c^{(1)},\eta^{(0)})$ is the linearized Babenko operator computed at the profile $\eta^{(0)}$ for the speed $c^{(1)}$. Once the nonlinear terms in $\delta \eta$ are neglected, the approximate equation for $\delta \eta$ yields
\begin{align}
    \mathcal{L}(c^{(1)},\eta^{(0)}) \delta\eta = - \mathcal{S}(c^{(1)},\eta^{(0)}), \label{fminres}
\end{align}
which is solved in the Fourier space by means of the minimum residual method~\cite{saad1992numerical} (MINRES). The linearized Babenko's operator $\mathcal L$ is self-adjoint in $L^2(\mathbb{T})$, but it is not positive definite. This makes MINRES the preferred method of computing solutions for the correction term (\ref{fminres}). Once the linear system in~\eqref{fminres} is solved, the approximate solution is updated via $\eta^{(0)} \to \eta^{(0)} +\delta \eta$, and the new $\delta \eta$ is found from (\ref{fminres}) with updated $\eta^{(0)}$. 
This algorithm is repeated until a convergence criterion $\|S(c^{(1)},\eta^{(0)} + \delta \eta)\|_{L^2} \leq 10^{-28}$ is reached, at which step we assign $\eta^{(1)} = \eta^{(0)} + \delta \eta$ for this value of $c^{(1)}$. For variable precision arithmetic, the GNU MPFR~\cite{fousse2007mpfr} and GNU MPC~\cite{mpc} libraries are used, and the fast Fourier Transform (FFT) C library is written based on~\cite{press2007numerical}.
The convergence rate of Fourier series is improved by means of auxiliary conformal mapping based on Jacobi elliptic function, see ~\cite{hale2009conformal} and applications of this method in ~\cite{dyachenko2022almost}.

\begin{figure}[htb!]
	\centering
	\includegraphics[width=0.425\linewidth]{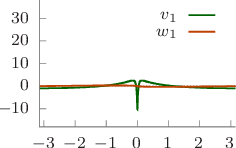}
    \includegraphics[width=0.425\linewidth]{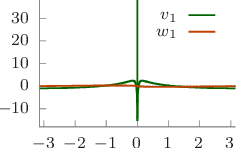}
	\includegraphics[width=0.425\linewidth]{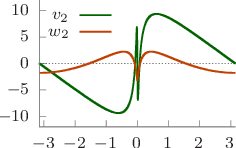}
    \includegraphics[width=0.425\linewidth]{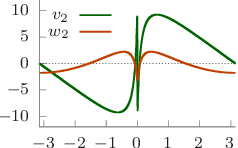}
    \includegraphics[width=0.425\linewidth]{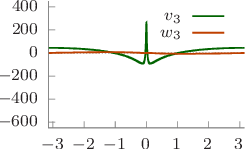}
    \includegraphics[width=0.425\linewidth]{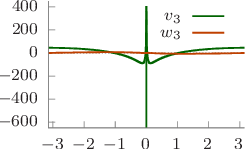}
	\caption{The generalized eigenvectors $(v_1,w_1)$, $(v_2,w_2)$ and $(v_3,w_3)$ defined via equations~\eqref{gen-Bab-eq} with $a_1 = 1$, $a_2 = 0$, \eqref{gen-Bab-eq-2-red} and \eqref{gen-Bab-eq-3} (top to bottom) for the first two critical points of the Hamiltonian at $s_1 = 0.13660354990$ (left) and $s_2 = 0.14079654715$ (right).}
	\label{fig:gen_vectors}
\end{figure}

\begin{figure}[htb!]
    \centering
    \includegraphics[width=0.495\linewidth]{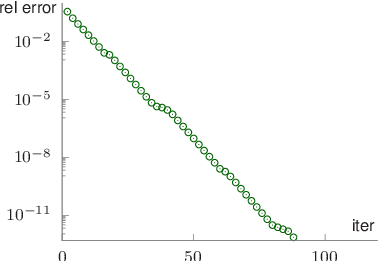}
    \includegraphics[width=0.495\linewidth]{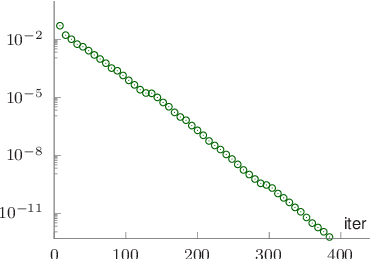}
    \caption{An example of numerical convergence of the iterative method for $(v_2,w_2)$ for the Stokes waves at $s_1 = 0.13660355$ (left) and $s_2 = 0.1407965471$ (right). The relative $L^2$ norm of the residual is shown versus the iteration number.  }
    \label{fig:numconv}
\end{figure}

\subsection{Generalized eigenfunctions}

We are construction a sequence of (generalized) eigenfunctions in the proof of Theorem \ref{theorem-crit}.  The algorithm is written in double precision arithmetic. The key element in determining generalized eigenfunctions is in solving linear inhomogeneous systems with the preconditioned MINRES, for which a symmetric positive definite preconditioner is defined by the strictly positive operator $(1 + c^2 K)$ to improve convergence rate. 

Figure~\ref{fig:gen_vectors} shows the three generalized eigenfunctions of the Jordan chains 
$(v_1,w_1)$, $(v_2,w_2)$, and $(v_3,w_3)$ defined via the linear inhomogeneous systems~\eqref{gen-Bab-eq} with $a = 1$, $a_2 = 0$, \eqref{gen-Bab-eq-2-red}, and~\eqref{gen-Bab-eq-3} for the Stokes wave at the first two extrema $s_1$ and $s_2$ of the Hamiltonian (see also Table~\ref{tab:wave_params}).

Figure~\ref{fig:numconv} illustrates the convergence rate of the preconditioned MINRES for the system~\eqref{gen-Bab-eq-2-red} to find the eigenfunction $(v_2,w_2)$ for Stokes waves at $s_1$ and $s_2$. The number of Fourier modes to represent the Stokes wave and the (generalized) eigenfunctions on a uniform grid is $N = 8192$ at $s_1$, and $N = 262144$ for $s_2$.

\begin{figure}[htb!]
	\centering
	\includegraphics[scale=1.25]{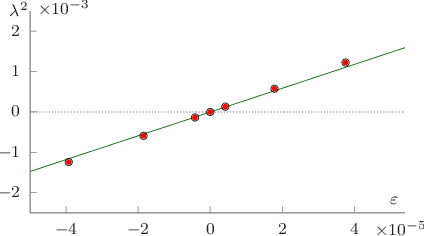}
    \includegraphics[scale=1.25]{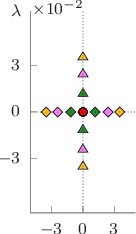}
    \includegraphics[scale=1.25]{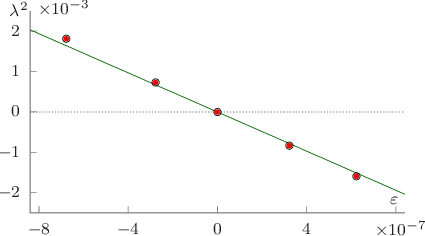}
    \includegraphics[scale=1.25]{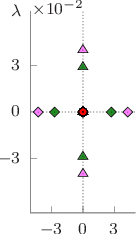}
    \caption{Top left shows $\lambda^2(\varepsilon)$ obtained from numerical solution of the stability problem~\eqref{spec-Bab-eq} (red dots), and evaluating the expansion~\eqref{halfseries} (green line) with $\lambda_1^2 = 29.4871$ ($\mathcal{B} = 11.01822$); and top right shows a pair of real eigenvalues appearing from a collision of two imaginary eigenvalues at $s_{1}$ described in~\eqref{halfseries} with $c-c_1 = \varepsilon=     -3.93\times10^{-5}$, $\varepsilon = -1.85\times10^{-5}$, $\varepsilon = -4.23\times 10^{-6}$ (gold, orchid and green triangles respectively), $\varepsilon=0$ (red circle), $\varepsilon = 4.17\times10^{-6}$, $\varepsilon = 1.78\times10^{-5}$ and $\varepsilon=3.75\times10^{-5}$ (green, orchid and gold diamonds respecively). Bottom row shows the same quantities at the second extremum at $s_2$ with $\mathcal{B} = 10.96232$.}
	\label{fig:det-s}
\end{figure}

\subsection{Finding eigenvalues of the stability problem} 

Eigenvalues of the stability problem can be found from numerically solving the eigenvalue problem~\eqref{spec-Bab-eq}, or equivalently, the quadratic pencil problem
\begin{align*}
    \left[\mathcal{M}^*K^{-1}\mathcal{M} \lambda^2 - 2c\mathcal{H}\lambda -\mathcal{L} \right] v = 0, 
\end{align*}
where $K^{-1}$ is defined under the orthogonality condition 
$\langle (1+2K\eta),v \rangle = 0$, see (\ref{orth-conditions}).
The quadratic pencil problem was used in~\cite{dyachenko_semenova2022} for co-periodic perturbations and in~\cite{dyachenko2023quasiperiodic,DDS2024} for subharmonic perturbations via the Bloch-Floquet theory (see also~\cite{deconinck2006computing} for the numerical Hill method). A new, $j^{th}$ pair of eigenvalues collide at each extrema of the Hamiltonian, $s_j$. The collision occurs at the origin in the spectral plane, and the eigenvalues become real as shown in Figure~\ref{fig:det-s} (right panels). It is convenient to show the square of eigenvalue $\lambda^2(\varepsilon)$ as a function of $\varepsilon=c-c_0$ and compare it to the expression~\eqref{normal-form-eq} to cross validate theory and numerics in Fig.~\ref{fig:det-s} (left panels). The direct computation of eigenvalues is obtained via the shift-and-invert method. 

It is interesting to note that the values of the coefficient $\mathcal{B}$ in the normal form (\ref{normal-form-eq}) are surprisingly close at both extrema of the Hamiltonian, see Table \ref{tab:wave_params}. The coefficient $\mathcal{B}$ is uniquely defined in the Puiseux expansion (\ref{halfseries}). A further investigation is needed to check if this behaviour is universal for all extremal points of the Hamiltonian.

\section{Conclusion}
\label{sec-6}

We summarize the main outcome of this work. We have used conformal variables for the two-dimensional Euler's equation in an infinitely deep fluid and computed the normal form for the zero eigenvalue bifurcaton 
of the Stokes waves with respect to co-periodic perturbations. The zero eigenvalue bifurcation occurs 
at every extremal point of the Hamiltonian or, equivalently, the horizontal momentum. The coefficient of the normal form computed numerically shows that the new unstable eigenvalues emerge in the direction of the 
increasing steepness of the Stokes wave. 

This work opens the road to analytic understanding of bifurcations of the unstable spectral bands in the modulational instability of Stokes waves by using the Bloch--Floquet theory. Numerical results have been computed recently in \cite{deconinck2022instability,DDS2024} and show interesting transformations of the spectral bands 
when the real unstable eigenvalues bifurcate in the space of anti-periodic and co-periodic perturbations. The figure-$8$ instability is replaced by the figure-$\infty$ instability at the co-periodic instability bifurcation, 
and this transformation is described by the normal form which extend the normal form derived here by the parameter of the Bloch--Floquet theory along the spectral bands. The analytical proof of this transformation is currently in progress. The recent work \cite{CuiP2025} describes a similar transformation of the figure-$8$ instability to the figure-$\infty$ instability in the local model of the focusing modified Korteweg--de Vries equation obtained via integrability of the model.  

\vspace{0.2cm}

{\bf Acknowledgement.} A part of this work was done while the second author attended the INI program ``Emergent phenomena in nonlinear dispersive waves" in Newcastle, UK (July--August, 2025). 

\bibliographystyle{amsplainabbrv}
\bibliography{main}

\end{document}